\documentclass[letterpaper, 10 pt, conference]{ieeeconf}  
\usepackage{amssymb}
\usepackage{amsmath}
\usepackage{algorithm}
\usepackage{algorithmic}
\usepackage{colortbl}

\newtheorem{assumption}{Assumption}

\newtheorem{remark}{Remark}
\newtheorem{theorem}{Theorem}
\newtheorem{lemma}{Lemma}

 \usepackage{graphicx}
 \usepackage {subcaption}
  \usepackage{cite}
  \makeatletter
  \let\NAT@parse\undefined
  \makeatother
  \usepackage{hyperref}
  \hypersetup{hypertex=true,
  	colorlinks=true,
  	linkcolor=blue,
  	anchorcolor=blue,
  	citecolor=blue}
  
\IEEEoverridecommandlockouts                              

\title{\LARGE \bf Distributed Optimization Method Based On Optimal Control}

\author{ Ziyuan Guo, Yue Sun,  Yeming Xu, Liping Zhang, and Huanshui Zhang, Senior Member, IEEE
\thanks{This work was supported by  the Original Exploratory Program Project of National Natural Science Foundation of China (62450004),  the Joint Funds of the National Natural Science Foundation of China (U23A20325),  and the Major Basic Research of Natural Science Foundation of Shandong Province (ZR2021ZD14). (Corresponding author: Huanshui Zhang).
}
\thanks{Ziyuan Guo, Yeming Xu and Liping Zhang are with the College
	of Electrical Engineering and Automation, Shandong University of Science
	and Technology, Qingdao, 266590, China (e-mail: skdgzy@sdust.edu.cn; ymxu2022@163.com; lpzhang1020@sdust.edu.cn).
	      }
\thanks{Yue Sun is  with the School of Control Science and
	      	Engineering, Shandong University, Jinan, Shandong 250061, China (email: sunyue9603@163.com).
	      }
\thanks{Huanshui Zhang is with the College of Electrical Engineering and
	Automation, Shandong University of Science and Technology, Qingdao
	266590, China, and also with the School of Control Science and
	Engineering, Shandong University, Jinan, Shandong 250061, China (email: hszhang@sdu.edu.cn).
}
}

\begin{document}

\maketitle
\thispagestyle{empty}
\pagestyle{empty}

\begin{abstract}
In this paper, a novel distributed optimization
framework has been proposed. The key idea is to convert
optimization problems into optimal control problems where
the objective of each agent is to design the current control
input minimizing the original objective function of itself and
updated size for the future time instant. Compared with the
existing distributed optimization problem for optimizing a sum
of convex objective functions corresponding to multiple agents, we
present a distributed optimization algorithm for multi-agents system  based on the results from the
maximum principle. Moreover, the convergence and superlinear
convergence rate are also analyzed stringently. 

\end{abstract}
\section{INTRODUCTION}
Optimization problem which involves maximizing or minimizing the objective function is ubiquitous in science and engineering \cite{Belegundu2000}.
In the traditional centralized optimization control, it requires a centralized manner to collect and deal with the information received from all other agents.
With increasing penetrations of the networked control systems, the division of labor and cooperation of multi-agent systems plays an increasingly important role in daily life.
Meanwhile, distributed optimization problems arise accordingly
in which several autonomous agents collectively try to achieve a global objective
and compared with the centralized optimization, distributed algorithms also have the potential to respect privacy of data, measurements, cost functions, and constraints, which becomes increasingly important in practical applications such as intelligent traffic system \cite{Eini2019}, electric power systems \cite{Molzahn2017}, formation control\cite{Olshevsky2010}, resource allocation \cite{Xue2006} and so on.

Distributed optimization has gained growing renewed interest and there are various distributed algorithms proposed in the literature.
The dual decomposition, which is the early class of distributed optimization algorithms presented in \cite{Wan2009}, typically relied on choosing the step size to ensure convergence.
Due to the inexpensive computation cost of (sub)gradients, the gradient descent algorithms is also a popular tool to deal with the distributed optimization.
\cite{Nedic2009} considered a unconstrained distributed computation model for optimizing a sum of convex objective functions corresponding to multiple agents
where each agent performed a consensus step and then a descent
step along the local (sub)gradient direction of its own convex objective function.
%
For the constrained distributed optimization problem,
\cite{Nedic2010} focused on the cooperative control problems where the values of agents are constrained to lie in closed convex sets and developed distributed optimization algorithms for problems.
To accelerate the convergence of (sub)gradient method, an effective distribution algorithm called Newton method is taken into consideration \cite{Islamov2021, Wei2013a, Wei2013b}, which has a superlinear rate in the convergence rate.
In particular, \cite{Zhang2022} proposed a distributed adaptive Newton algorithm with a global quadratic convergence rate where each node of a peer-to-peer network minimizes a finite sum of objective functions by communicating with its neighboring nodes.
The procedure proposed in \cite{Varagnolo2016} let agents distributedly compute and sequentially
update an approximated Newton-Raphson direction by means of suitable average consensus ratios.
However, as point out in \cite{Zhang2024}, although subgradient and Newton methods are favorable, powerful and widely used in distributed optimization algorithms, the choice of step size and the singularity of Hessian matrix are crucial to the algorithms.

Motivated by \cite{Zhang2024}, a novel distributed optimization algorithm is proposed from the  viewpoint of the optimal control problem. To be specific, there exists control input of each agent who is the updated size of each iteration. The objective of each agent is to design current control input to minimize the objective function of itself and updated size for the future time instant.
Through a simple transformation, the sum of the original objective function of each agent is constructed.
Based on the maximum principle, the distributed optimization method is derived compared with the existing distributed optimization problem for optimizing a sum of convex objective functions.
To show the superiority of the algorithm, the superlinear convergence rate is proved.


The outline of this paper is given below. Compared with the
distributed optimization problem, the optimal control problem
is derived in Section II. The distributed optimization algorithm
based on the maximum principle is proposed in Section
III. The rate of the convergence is analysed in Section IV.
Numerical examples are given in Section V. Conclusion is
arranged in Section VI.

{\bf Notation}: Throughout the paper,  $A^T$ stands for the transposition of
matrix $A$. $R^n$ denotes the $n$-dimensional real vector space. $I$ is a unit matrix with appropriate dimension. For a symmetric matrix $M$, $M > 0(\geq 0)$ means
that $M$ is a positive definite (positive semi-definite) matrix. $\rho(M)$ represents the spectral radius of $M$.  $\|\cdot\|$ denotes 2-norm of vectors and $\|\cdot\|_m$ is the induced norm of matrices. $[a, b]$ represents all integers from integer $a$ to integer $b$.

\section{PROBLEM FORMULATION}
\subsection{Graph theory}
The communication graph is denoted by a directed graph $\mathcal{G=(V,E,A)}$, where $\mathcal{V}=\{1,2, \ldots ,n\}$ is a set of vertices(nodes), $n$ is the number of agents satisfying $n \geq 2$, $\mathcal{E} \subseteq \mathcal{V} \times \mathcal{V}$ is the set of edges, and the
weighted matrix $\mathcal{A} = (a_{ij})_{n\times n}$ is a non-negative matrix
for adjacency weights of edges, $a_{ij} \neq 0$ if and only if $(i, j) \in \mathcal{E}$. The graph $\mathcal{G}$ is said to be balanced if the sum of the interaction weights from agent $j$ to agent $i$ are equal, i.e., $\sum\limits_{j=1}^{n}a_{ij}=\sum\limits_{j=1}^{n}a_{ji}, \forall i\in \mathcal{V}$. Moreover, $\mathcal{N}_i = \{j \in \mathcal{V}|{(i, j)} \in \mathcal{E}\}$ is used to represent the neighbor set of agent $i$. For the topology $\mathcal{G = (V, E, A)}$, a path of length $r$ from node $i_1$ to node $i_r+1$ is a sequence of $r + 1$ distinct
nodes ${\{i_1\ldots , i_r+1\}}$ such that $(i_q, i_q+1) \in \mathcal{E}$ for $q = 1, \ldots ,r$.
If there exists a path between any two nodes, then $\mathcal{G}$ is said to be strongly connected.
Here we make the following assumptions for the communication graph.
\begin{assumption}\label{assum1}
The directed graph $\mathcal{G}$ is balanced and is strongly connected.	
\end{assumption}

{Noting that } in the distributed coordination control, conditions in Assumption 1 play important roles in ensuring {multi-agent systems (MAS)} to reach consensus {\cite{Nedic2009} \cite{blondel2005convergence}}.
\subsection{Distributed optimization}

Consider an MAS consisting of $n$ agents {with optimization problem}, labeled by set $\mathcal{V} =\{1, \dots , n\}$ where agents communicate  the local state information with their neighbors via directed graph
$\mathcal{G}$. Only  agent $i \in \mathcal{V} $ knows the function $f_{i}:{\mathbb{R}^n} \to \mathbb{R}${, which is  twice continuously differentiable, possibly non-convex.  The traditional unconstrained optimization problem is given by}
\begin{equation}
		\mathop {\rm minimize}\limits_{x \in {\mathbb{R}^n}} \{f(x)\overset{\text{def}}{=}\frac{1}{n}\sum\limits_{i=1}^{n}f_{i}(x)\}. \label{1}
\end{equation}

Different from the traditional distributed optimization method,
we transform the task of finding solutions to problem (\ref{1}) into  updating of the state sequence within an optimal control problem.
To be specific, consider the following discrete-time linear time-invariant system for each  agent
\begin{equation}
	{x_{i}(k+1)} = x_{i}(k) + u_{i}(k), \label{3}
\end{equation}
where {$x_i(k)\in R^n$} is  state with the initial value $x(0)$, {$u_i(k)\in R^n$} is the  control input which needs to be further specified later. The $i$ represents the $i$th agent, $i \in {\{1,2,\dots n\}}$.

Correspondingly, the cost function of each agent satisfies
\begin{eqnarray}\label{02}
  J_i &\hspace{-0.8em}=&\hspace{-0.8em}\sum\limits^{N}_{k=0}\Big(\sum\limits_{j\in \mathcal{N}_i}e^T_{ij}(k)Q_ie_{ij}(k)+u^T_i(k)R_iu_i(k)+f_i(x_i(k))\Big)\nonumber\\
  &\hspace{-0.8em}&\hspace{-0.8em}+\sum\limits_{j\in \mathcal{N}_i}e^T_{ij}(N+1)H_ie_{ij}(N+1)+f_i(x_i(N+1)),
\end{eqnarray}
with $e_{ij}(k)=x_i(k)-x_j(k)$
where $Q_i\geq0$, $H_i\geq0$ and $R_i>0$.

\begin{remark}
Compared with the traditional method, we convert the optimization problem into an optimal control problem,
where the objective of each agent is to design the current control input minimizing the sum of the original objective function and updated size for the future time instant.
On the one hand, the problem studied in this paper includes the traditional results with multi-agents system, if we let $f_i(x_i(k))=f_i(x_i(N+1))=0$ while the system (\ref{1}) can be modified as $x_i(k+1)=Ax_i(k)+B_iu_i(k)$.
On the other hand, in the existing optimization algorithms to handle the problem (\ref{1}), there is a typical distributed first-order gradient descent algorithm \cite{ruder2016overview}:
\begin{equation}
{x_{i}(k+1)} = \sum\limits_{j=1}^{n}a_{ij}x_{j}(k) - \eta(k)\nabla f_{i}({x_i}(k)),\label{2}
\end{equation}
where  $\eta(k)$ is the step size.
Noting that in convex optimization, it has been proven that using distributed algorithm (\ref{2}), {the global optimal solution for problem (\ref{1}) can be obtained}. However, distributed algorithm (\ref{2}) has limitations in terms of selecting the step size. The step size $\eta(k)$ in distributed algorithm (\ref{2}) should theoretically satisfy the following two conditions:
\begin{equation}
lim_{k \to \infty} \eta(k)=0,\label{78}
\end{equation}
\begin{equation}
\sum\limits_{k=1}^{\infty}\eta(k)=\infty.\label{77}
\end{equation}
In practical applications, the common practice for selecting the $\eta$ is to start with a small constant value. Alternatively, even if a time-varying step size $\eta(k)$ is chosen to satisfy the two conditions (\ref{78})-(\ref{77}), the results obtained from distributed algorithm (\ref{2}) are evidently influenced by the $\eta(k)$.
The distributed algorithm (\ref{2}) has limitations in terms of selecting the step size and
the results obtained from distributed algorithm (\ref{2}) are evidently influenced by the $\eta(k)$.
\end{remark}

{ Based on the discussion above, we are in position to derive the distributed optimization algorithm from the viewpoint of the optimal control.}
{ For convenience of future discussion, some symbols are denoted below.}
\begin{eqnarray*}
x(k)&\hspace{-0.8em}=&\hspace{-0.8em}[x^{T}_1(k),...,x^{T}_n(k)]^T, \quad u(k)=[u^{T}_1(k),...,u^{T}_n(k)]^T,\\
Q&\hspace{-0.8em}=&\hspace{-0.8em}diag\{Q_1,...,Q_n\},\quad R=diag\{R_1,...,R_n\},\\
H&\hspace{-0.8em}=&\hspace{-0.8em}diag\{H_1,...,H_n\}, \quad F({x(k)})=\sum\limits_{i = 1}^nf_i({x_{i}(k)}).
\end{eqnarray*}
%
Let $\delta_i(k)$ be the  errors set of the $i$th agent, then $e(k)=[\delta^T_1(k),...,\delta^T_n(k)]$. In this way, we have
\begin{eqnarray*}
e(k+1)=e(k)+\sum\limits_{i = 1}^nB_iu_i(k)=e(k)+Bu(k),
\end{eqnarray*}
where $B_i$ is the column vector composed by $0, 1, -1$, satisfying $B=[B_1,...,B_n]$ and $B^TB=L$ is a Laplacian matrix.
To this end,  (\ref{3}) and (\ref{02}) can be written with  expanded forms as
\begin{eqnarray}\label{globalsystem}
x(k+1)=x(k)+u(k),
\end{eqnarray}
and
\begin{equation}
		\begin{array}{l}
		\vspace{0.15cm}{\rm{ }}\mathop J_N=\sum\limits_{k = 0}^N (\sum\limits_{i = 1}^n(\sum\limits_{j \in N_{i}}e^{T}_{ij}(k)Q_{i}e_{ij}(k)+ u_i(k)^\mathrm {T} R_{i}u_i(k)\\
		\vspace{0.15cm}\quad\quad +f_{i}({x_{i}(k)}))) + \sum\limits_{i = 1}^n(\sum\limits_{j \in N_{i}}e^{T}_{ij}(N+1)H_{i}e_{ij}(N+1)\\
		\quad\quad+f_i({x_{i}(N+1)}))\\
		\vspace{0.15cm}\quad\ =\sum\limits_{k = 0}^N (e^{T}(k)Qe(k)+ u(k)^\mathrm {T} Ru(k)) +F({x(k)})\\
		\vspace{0.15cm}\quad\quad + e^{T}(N+1)He(N+1)
		+F({x(N+1)}),
\end{array}\label{4}
\end{equation}
respectively. The optimal control problem to be solved in this section is given.

\emph{Problem:}
Find the optimal control $u(k)$  minimizing the long-term cost $J_N$  and subject to (\ref{globalsystem}).

\begin{remark}
In our recent work \cite{Zhang2024}, we implemented a centralized optimization algorithm, {which causes us to design the system and cost function (\ref{globalsystem})-(\ref{4}).}  Naturally, we considered the possibility of designing  a distributed algorithm {based on the centralized optimization algorithm from the maximum principle}. It should be noted that if $F({x(k)})=\sum\limits_{i = 1}^nf_i({x_{i}(k)})=0$, (\ref{4}) can be reduced to a cost function of consensus problem based on optimal control  in our recent work \cite{zhang2023distributed}.
\end{remark}
\section{Distributed Optimization Method Using Optimal Control}
In this section, we use optimal control theory to solve the Problem. This inspires us to develop an
optimization algorithm which can be implemented into a distributed manner. Additionally, a variant of the
algorithm is also proposed to balance the number of iterations and communications.
\subsection{Algorithm Development}
Following from \cite{Zhang2012}, applying Pontryagins maximum principle to the system (\ref{globalsystem}) with the cost function (\ref{4}), the following costate equation and equilibrium condition are obtained:
\begin{equation}
	0= Ru(k) + \lambda(k), \label{ph}
\end{equation}
\begin{equation}
	{\lambda(k-1)} = B^TQe(k) + \nabla f(k)+\lambda(k), \label{costate}
\end{equation}
with the terminal value ${\lambda(N)} = B^THe(N+1) + \nabla f(N+1)$, where $\nabla f(k)=[\nabla f_1(x_1(k),...,\nabla f_n(x_n(k)))]^T$.

To derive the analytical solution from the FBDEs (\ref{globalsystem}), (\ref{ph})-
(\ref{costate}), define the Riccati equation
\begin{equation}
	{P(k)} = Q+P(k+1)-P(k+1)B\Gamma^{-1}(k)B^TP(k+1), \label{riccati}
\end{equation}
with the terminal value ${P(N+1)} = H$, where $	{\Gamma(k)} = R+B^TP(k+1)B$.

The analytical expression to the controller is given.
\begin{lemma}\label{Lemma1}
	If the Riccati equation (\ref{riccati}) admits solution such that ${\Gamma(k)}$ is invertible, then the control satisfies
	\begin{equation}
		{u(k)} = -\Gamma^{-1}(k)[B^TP(k+1)e(k)+\sum\limits_{l = k + 1}^{N + 1}M(l)\nabla f(l)], \label{uk}
	\end{equation}
    where {$M(l)=M(l-1)R(R+B^TP(i)B)^{-1} $ for $l> k+1$ and $M(l)=I$ for $l=k+1$}. Moreover, the costate equation is derived as
	\begin{equation}
		{\lambda(k-1)} = B^TP(k)e(k) + \sum\limits_{l = k}^{N + 1}M(l)\nabla f(l). \label{costate1}
	\end{equation}
\end{lemma}

\begin{proof}
For $k = N$, adding the terminal value $\lambda(N)$ into (\ref{ph}), there holds
\begin{equation}
	\begin{split}
		0 &= Ru(N)+B^TH[e(N)+Bu(N)]+\nabla f(N+1) \\
		&=[R+B^THB]u(N)+B^THe(k)+\nabla f(N+1),\label{8}
	\end{split}
\end{equation}
i.e., $u(N)$ satisfies (\ref{uk}) for $k = N$. Accordingly, substituting
(\ref{8}) into (\ref{costate}), the costate $\lambda(N-1)$ is calculated as

\begin{equation*}
	\begin{split}
		&\lambda(N-1) \\
=& B^TQe(N)\!+\!B^TH[e(N)\!+\!Bu(N)]\!+\!\nabla f(N\!+\!1)\!+\!\nabla f(N)\\
		=&B^T[Q+H]e(N)-B^THB\Gamma^{-1}(k)[B^TP(k+1)e(k)\\
		& +\nabla f(N+1)]+\nabla f(N+1)+\nabla f(N)\\
		=&B^TP(N)e(N)+M(N)\nabla f(N+1)+\nabla f(N),
	\end{split}
\end{equation*}
which is exactly (\ref{costate1}) for $k = N$, where $P(N)$ satisfying (\ref{riccati})
has been used in the last equality.

Assume that (\ref{costate1}) is established for $k\geq s+1$, we will show it also holds for $k=s$. For $k=s$, adding $\lambda(s)$ into (\ref{ph}), one has

\begin{equation}
0=Ru(s)+B^TP(s+1)[e(s)+Bu(s)]+ \sum\limits_{l = s + 1}^{N + 1}M(l)\nabla f(l),\label{13}
\end{equation}
i.e., $u(s)$ is such that
	\begin{equation}
		{u(s)} = -\Gamma^{-1}(s)[B^TP(s+1)e(s)+\sum\limits_{l = s + 1}^{N + 1}M(l)\nabla f(l)]. \label{us}
	  \end{equation}
	
	  Adding it into (\ref{costate}) for $k = s$, it derives
	  \begin{equation*}
	  	\begin{split}
	  		\lambda(s-1) &= B^TQe(s)+B^TP(s+1)[e(s)+Bu(s)]\\
	  		&+\sum\limits_{l = s + 1}^{N + 1}M(l)\nabla f(l)+\nabla f(s) \\
	  		&=B^TP(s)e(s)+\sum\limits_{l = s}^{N + 1}M(l)\nabla f(l),
	  	\end{split}
	  \end{equation*}
	  which is exactly (\ref{costate1}). This completes the proof.
\end{proof}

\begin{remark}
It should be noted that the analytical solution $u(k)$ obtained in Lemma 1 contains the future information, i.e., it depends on  $\{\nabla f(i), i \in [k + 1,N + 1]\}$, which is difficult to realize in practical life, thus a modified algorithm is presented.
\end{remark}

Before proposing an algorithm based on optimal control obtained in Lemma \ref{Lemma1}, we first show an interesting phenomenon: the controller (\ref{uk}) derived from optimal control theory utilizes the average value of the derivative of each agent's local function, referred to as the ``average gradient'' hereafter. {That is to say, the matrix $M(l)$ has the property of taking the average value, which is proved below. It should be noted that d}istributed optimization algorithms that employ the average gradient have been explored in {\cite{qu2017harnessing}}. Our derived results theoretically demonstrate the rationality and correctness of adopting average gradient.
{ The property of $M(l)$ defined in Lemma \ref{Lemma1} is given below.}
\begin{lemma}
{$M(l)$ defined in Lemma \ref{Lemma1} satisfies}
$lim_{i \to \infty}M(l)=\frac{1}{n}\boldsymbol {1}\boldsymbol {1^T}$.
\end{lemma}

\begin{proof}
If the Riccati equation (\ref{riccati}) admits  a stable solution $P$ {when $k\rightarrow \infty$}, then we adopt this stable solution in $M(l)$. In order to prove the convergence of $M(l)$,  the following system is constructed firstly
\begin{equation}
	\begin{split}
	{w(k+1)} &= R(R+B^TPB)^{-1}w(k)\\
	         &= (R(R+B^TPB)^{-1})^kw(0), \label{wk}
	\end{split}
\end{equation}
where $w(k) \in {\mathbb{R}^n}$. {In this case, the key point is to prove} $lim_{k \to \infty} (R(R+B^TPB)^{-1})^k=\frac{1}{n}\boldsymbol {1}\boldsymbol {1^T}$. Utilizing the properties of the Riccati equation's solution and $B^TB=L$, rewrite the above  (\ref{wk}) with finesse
\begin{equation}
		\begin{split}
	{w(k+1)} &= (I+\bar L)^{-1}w(k)\\
	         &= w(k)-((I+\bar L)^{-1}\bar L)w(k), \label{alwk}
		\end{split}
\end{equation}
 where $\bar L=B^TPBR^{-1}$ is a Laplacian matrix.

 According to {\cite{olfati2004consensus}}, since the eigenvalues of matrix $\bar L$ are greater than or equal to zero, the eigenvalues of $(I+\bar L)$ are greater than or equal to $1$. Therefore, the eigenvalues of the inverse of $(I+\bar L)$ lie between $0$ and $1$, indicating that the above system  (\ref{wk}) is stable.

 When $k \rightarrow \infty$, the stable state $w$ of the system should satisfy $-((I+\bar L)^{-1}\bar L)w=0$, which simplifies to $-\bar Lw=0$. Referring to conclusion {\cite{olfati2004consensus}}, the solution of $-\bar Lw=0$ is $w^*=\frac{1}{n}\boldsymbol {1}\boldsymbol {1^T}w(0)$, so we have $lim_{k \to \infty} (R(R+B^TPB)^{-1})^k=\frac{1}{n}\boldsymbol {1}\boldsymbol {1^T}$, i.e., $lim_{i \to \infty}M(l)=\frac{1}{n}\boldsymbol {1}\boldsymbol {1^T}$. This completes the proof.
\end{proof}	

Lemma 2 indicates that (\ref{uk}) utilizes the average gradient. For the sake of simplifying notation, we will denote the average gradient as $g(l)$, i.e., $g(l)=M(l)\nabla f(l)$.
\subsection{Distributed Optimization Algorithm } \label{555}
In this subsection, we discuss the distributed optimization algorithm derived based on Lemma 1, which includes a centralized algorithm and a decentralized algorithm.
\begin{lemma}
	Under the condition of Lemma 1 and based on the optimal control obtained from it, there exists an iteration algorithm
	\begin{equation}
		\begin{split}
			{x(k+1)} &= x(k)+d_k(x(k))\\
			{d_l(x(k))}&= -(\Gamma_P+h(k))^{-1}[g(k)-\Gamma_Pd_{l-1}(x(k))], \\
			d_0(x(k))&=-(\Gamma_P+h(k))^{-1}(g(k)+B^TPe(k)).\label{dk}
		\end{split}
	\end{equation}
	where $l=1,\dots,k$ means the required number of cycles, $g(k)=(\frac{1}{n}\boldsymbol {1}\boldsymbol {1^T})\nabla f(k)$ is average gradient, $h(k)=diag\{h_1(k),...,h_n(k)\}$, $h_i(k)=\frac{1}{n}\sum\limits_{i = 1}^{n}\nabla^2 f_i(x_i(k))$, $\Gamma_P$ represents the stable solution $P$ of (\ref{riccati}) used for $\Gamma(k)$.
\end{lemma}
\begin{proof}
	For controller (\ref{uk}), noting that it exhibits non-causality because the future average gradient is used.  Fortunately, the problem of utilizing the future information has been previously addressed in our research  referring to { \cite{Zhang2024}} for the details. In accordance with the methodology presented in { \cite{Zhang2024}}, controller \eqref{uk} will transform into $d_k$ as depicted in \eqref{dk}.
\end{proof}

\begin{remark}
{ The connection between algorithm \eqref{dk} and our previous work presented in \cite{Zhang2024} will be discussed}. On the one hand, if each agent does not have a local function to optimize, {i.e., $f_i(x_i(k))=f_i(x_i(N+1))=0$}, it is evident that the terms $g(k)$ and $h(k)$ in \eqref{dk} will not exist. {That is to  say,} \eqref{dk} transforms into the following form:
	\begin{equation}
	\begin{split}
		&{x(k+1)} = x(k)+d_k\\
		&d_k=-(\Gamma_P)^{-1}B^TPe(k).\label{condk}
	\end{split}
\end{equation}
In this case, $d_k$ corresponds to the form of the controller $u_k$ from Lemma 1  in {\cite{zhang2023distributed}} when the matrix $A$ is the identity matrix. On the other hand, if we only consider a  centralized optimization problem, or if the states during the iteration process of each agent are same, then $e(k)$ and the related $P$ will not exist.  Moreover,  \eqref{dk} degenerates into the following form:
	\begin{equation}
	\begin{split}
		{x(k+1)} &= x(k)+d_k(x(k))\\
		d_l(x(k))&= -(R+h(k))^{-1}[g(k)-Rd_{l-1}(x(k))], \\
		d_0(x(k))&=-(R+h(k))^{-1}g(k).\label{ocp}
	\end{split}
\end{equation}
where $l=1,\dots,k$ means the required number of cycles. In this case, \eqref{ocp} corresponds to Algorithm I in the previous work { \cite{Zhang2024}}, where $g(k)$ and $h(k)$ correspond to the (average) gradient and Hessian matrix, respectively.
\end{remark}

Based on optimal control theory, we have derived a distributed optimization algorithm \eqref{dk} with a central computing node (in computer science, it is called as $parameter \  server$ \cite{soori2020dave} and we use this expression in Algorithm 1).  The algorithm presented in Algorithm 1 is called the Distributed Optimal Control Method with a Central Computing Node (DOCMC).  Now, we summarize the DOCMC algorithm from the perspective of each agent $i$.
\begin{algorithm}[H]
	\renewcommand{\algorithmicrequire}{\textbf{Input:}}
	\renewcommand{\algorithmicensure}{\textbf{Output:}}
	\caption{DOCMC}
	\label{alg1}
	\begin{algorithmic}[1]
		\STATE Initialization: Each agent checks whether it can communicate with the server, and sets $x_i(0) \in {\mathbb{R}^p}$
		\FOR {$k=0,1,...$}
				\FOR {each agent $i=1,...,n$}
		\STATE Compute local gradient $\nabla f_{i}({x_i}(k))$ and local hessian ${\nabla^2 f_{i}({x_i}(k))}$. Calculate the state error $e_{ij}(k)$ between itself and its neighbor.
		\STATE Send $\nabla f_{i}({x_i}(k))$,  $\nabla^2 f_{i}({x_i}(k))$ and $e_{ij}(k)$ to server.	
		       \ENDFOR	
	    \STATE For the server, compute $g(k)$, $h(k)$ and $d_0(x(k))=-(\Gamma_P+h(k))^{-1}(g(k)+B^TPe(k))$.
	      \FOR {$l=0,1,...,k-1$}
	      \item
	        ${d_{l+1}(x(k))= -(\Gamma_P+h(k))^{-1}[g(k)-\Gamma_Pd_{l}(x(k))]}$
	     \ENDFOR
	     \STATE Broadcast ${d_k(x_i(k))}$ to $i$-th agent, where $d_k(x_i(k))$ is the components corresponding to the $i-th$ agent in $d_k(x(k))$.
	     \STATE Each agent update $x_i(k+1)=x_i(k)+d_k(x_i(k))$
		\ENDFOR	
	\end{algorithmic}
\end{algorithm}

Utilizing the state error $e_{ij}$ with neighbors can significantly enhance the computational stability of the server.
In fact, if each agent adopts the same state, then \eqref{ocp} will be utilized in DOCMC. However, when the states of each agent cannot be guaranteed to be consistent due to
communication processes or the influence of individual agents,
relying solely on \eqref{ocp} may not be sufficient to achieve the optimization goals.
The existence of $e_{ij}$ ensures that each agent can achieve the optimization goals even when their states are not consistent (whether due to initially different states or deviations during the optimization process), as the algorithm incorporates an optimal consensus feedback control form based on $e_{ij}$.

Next, we will focus on describing the distributed  optimization algorithm. Firstly, reconsidering the DOCMC algorithm, the existence of $(\Gamma_P+h(k))^{-1}$ is a primary reason why the algorithm requires a server.  Motivated by our recent work {\cite{xu2024distributed}}, which replaced the inverse matrix with an adjustable scalar $\eta >0$, in the DOCMC algorithm,  $(\Gamma_P+h(k))^{-1}$ {  will be replaced} by $\eta >0$ in { the following discussion}. Due to the relationship   $(\Gamma_P+h(k))^{-1}\Gamma_P=I-(\Gamma_P+h(k))^{-1}h(k)$ {  based on (\ref{dk})}, we  propose the following algorithm
	\begin{equation}
	\begin{split}
		{x(k+1)} &= x(k)+\overline{d}_k(x(k))\\
		\overline{d}_l(x(k))&= -\eta g(k)+(I-\eta h(k))\overline{d}_{l-1}(x(k)), \\
		\overline{d}_0(x(k))&=-\eta(g(k)+B^TPe(k)),\label{ddk}
	\end{split}
\end{equation}
 where $l=1,\dots,k$ means the required number of cycles, the second-order information $h(k)$ is used in algorithm \eqref{ddk} on account of the optimization framework of optimal control. We refer to the above algorithm as the Distributed Optimization Algorithm based on Optimal Control (DOAOC).

{ 
\begin{remark}
{Now the connection between algorithm \eqref{ddk} and our previous work in {\cite{wang2024superlinear}} is discussed}. If we only consider a  centralized optimization problem, or if the states during the iteration process of each agent are the same, then $e(k)$ and the related $P$ will not exist.  {  In this case,} \eqref{ddk} { degrades} into the following form:
	\begin{equation}
	\begin{split}
		{x(k+1)} &= x(k)+\overline{d}_k(x(k))\\
    \overline{d}_l(x(k))&= -\eta g(k)+(I-\eta h(k))\overline{d}_{l-1}(x(k)), \\
    \overline{d}_0(x(k))&=-\eta g(k).\label{ocpl}
	\end{split}
\end{equation}
{  which is exactly Algorithm II in the previous work \cite{wang2024superlinear}.}
\end{remark}}

Now, {a distributed optimization algorithm for each agent $i$ to handling problem {\eqref{1}} from the DOAOC algorithm is given, referring}
to Algorithm 2 for details.
{  In the proposed distributed algorithm, as described in Steps 3 and 4,  each agent only communicates with its neighbors.}
Step 5 initializes the loop by computing the initial $\overline{d}^0_{i}(k)$. Different from {\cite{xu2024distributed}},  Step 7 only requires each agent to calculate $\overline{d}^{t+1}_{i}(k)$ independently.
\begin{algorithm}[H]
	\renewcommand{\algorithmicrequire}{\textbf{Input:}}
	\renewcommand{\algorithmicensure}{\textbf{Output:}}
	\caption{DOAOC}
	\label{alg1}
	\begin{algorithmic}[1]
		\STATE Initialization: Each agent $i$ requires $\eta$ and sets $x_i(0) \in {\mathbb{R}^p}$
		\FOR {$k=0,1,...$}
		\STATE Take any average consensus protocol to get $h_i(k)$ and $g_i(k)$.
		\STATE Each agent $i$ obtains $e_{ij}(k)$ by communicating with its neighbors $j \in { \mathcal{N}_i}$ and {  combining with $e_{ij}(k)$, agent $i$ records}  $\delta^T_i(k)$.
		\STATE Each agent $i$ calculates $\overline{d}^0_{i}(k)=-\eta(g_i(k)+B_i^TP_i\delta^T_i(k))$.
		\FOR {$l=0,1,...,k-1$}
		\STATE Each agent $i$ calculates
		
		$\overline{d}^{l+1}_{i}(k)= \overline{d}^l_{i}(k)-\eta g_i(k)-\eta h_i(k)\overline{d}^l_{i}(k)$.
		\ENDFOR
		\STATE Each agent $i$ updates $x_i(k+1)=x_i(k)+\overline{d}^{k}_{i}(k)$
		\ENDFOR	
	\end{algorithmic}
\end{algorithm}

\subsection{Dissussion of the DOCMC and DOAOC Algorithm} \label{123}
In this subsection, we discuss the compelling characteristics of DOCMC and DOAOC and some simplification techniques in practical use{ , which}
are summarized as follows:
\begin{itemize}	
	\item
Many distributed optimization methods commonly used in machine learning, including frameworks like TensorFlow, employ distributed algorithms with a central computing node. When considering a star network (master-slave network) structure \cite{soori2020dave} in the DOCMC algorithm (where $e_{ij}$ does not occur as each agent independently exchanges information with the server), simplification can be achieved using (21), where $g(k)$ represents the average derivative and $h(k)$ directly represents $\frac{1}{n}\sum\limits_{i = 1}^{n}\nabla^2 f_i(x_i(k))$.
	\item
To obtain the average gradient, Step 3 in Algorithm 2 is necessary. In practical applications, after evaluating the performance of the agents, a suitable consensus protocol can be chosen: if agents have ample memory, memory-intensive consensus protocols like the DSF algorithm \cite{zhang2022distributed} can be employed; if agents have limited memory, a finite-time average consensus protocol \cite{charalambous2018laplacian} can be used; and in cases where agent performance is poor, traditional consensus protocols \cite{Nedic2009} \cite{olfati2004consensus} can still be applied. $h_i(k)$ in DOAOC algorithm  can be obtained not only through the above consensus protocol but also can be approximated by first-order derivative difference. In fact, in our recent work \cite{wang2024superlinear}, Algorithm III utilizes a first-order derivative difference method to obtain the Hessian matrix.
\end{itemize}

\section{CONVERGENCE ANALYSIS} \label{222}
In this section, we conduct a convergence analysis of the proposed algorithms. The superlinear convergence rate demonstrates the superiority of our optimal control-based algorithms. The assumption provided in this subsection is instrumental for drawing our conclusion.
\begin{assumption}\label{assm2}
	The local objective functions $f_i, i=1,...,n$ are twice continuously differentiable, and there exist constants $0<m_1\le m_2<\infty$ such that for any $x \in {\mathbb{R}^n}$, $m_1I\preceq \nabla^2f_i(x)\preceq m_2I$
\end{assumption}

\begin{remark}
Assumption \ref{assm2} is standard in the convergence analysis , see {\cite{boyd2004convex}}. The lower bound { $m_1$} of $\nabla^2f_i(x)$ implies that the local objective function $f_i$ is strongly convex, and thus the objective function { $f_i$} has a unique global minimum point. The upper bound $m_2$ of $\nabla^2f_i(x)$ means that the local objective function $f_i$ is smooth.
\end{remark}

Let $x^*$ be the minimum point of $F(x)$, where $g^*$ represents $g(x^*)$ and $h^*$ represents $h(x^*)$, now we can proceed with the convergence analysis of our algorithms.

For DOCMC algorithm, we can compactly write the update iteration in \eqref{dk} as
\begin{equation}
	\begin{split}
		{x(k+1)} &= x(k)-[I-((\Gamma_P+h(k))^{-1}\Gamma_P)^{k+1}]\\
		&\times h(k)^{-1}g(k)-[(\Gamma_P+h(k))^{-1}\Gamma_P]^{k+1}\\
		&\times \Gamma_P^{-1}B^TPe(k).\label{sdk}
	\end{split}
\end{equation}
\begin{lemma}
	Under the { condition} of Assumption \ref{assm2}. DOCMC algorithm  converges if {  the weighted matrix $R$} is selected appropriately.		
\end{lemma}
\begin{proof}
	First{ ly}, we prove that $x(k)$ converges to $\overline{x}(k)$, where $\overline{x}(k)=(\frac{1}{n}\boldsymbol {1}\boldsymbol {1^T})x(k)$.
	\begin{equation}
		\begin{split}
			&\rVert x(k+1)-\overline{x}(k+1) \rVert \\
			&=\rVert x(k)-[I-((\Gamma_P+h(k))^{-1}\Gamma_P)^{k+1}] h(k)^{-1}g(k)\\
			&-[(\Gamma_P+h(k))^{-1}\Gamma_P]^{k+1}R^{-1}B^TPe(k)-\overline{x}(k)\\
			&+[I-((\Gamma_P+h(k))^{-1}\Gamma_P)^{k+1}] h(k)^{-1}g(k) \rVert \\
			&=\rVert x(k)-[(\Gamma_P+h(k))^{-1}\Gamma_P]^{k+1}\Gamma_P^{-1}B^TPe(k)-\overline{x}(k)\rVert\\
			&=\rVert [I-((\Gamma_P+h(k))^{-1}\Gamma_P)^{k+1}\Gamma_P^{-1}\bar L]x(k)-\overline{x}(k)\rVert.\label{pfav}
		\end{split}
	\end{equation}
	Let $\epsilon= ((\Gamma_P+h(k))^{-1}\Gamma_P)^{k+1}\Gamma_P^{-1}$. 
{ Recalling} that $\Gamma_P>0$ and $
	h(k)>0$,  { it} is { obtained immediately}
from {\cite{wang2006matrix}} that every eigenvalue of $(\Gamma_P+h(k))^{-1}\Gamma_P$ is positive and less than 1. Accordingly, all eigenvalues of $((\Gamma_P+h(k))^{-1}\Gamma_P)^{k+1}$ are positive and its spectral radius is less than 1. It should be noted that $((\Gamma_P+h(k))^{-1}\Gamma_P)^{k+1}$ is a symmetric matrix{ , which} can be proved in {\cite{wang2024superlinear}}. Therefore,  $W=I-\epsilon \bar L$  satisf{ ies} the properties of $P_\epsilon$ in {\cite{olfati2004consensus}}. According to {\cite{qu2017harnessing}}, then we have
	\begin{equation}
		\begin{split}
			\rVert x(k+1)-\overline{x}(k+1) \rVert
			&= \rVert (W-\frac{1}{n}\boldsymbol {1}\boldsymbol {1^T})(x(k)-\overline{x}(k)) \rVert\\
			&\leq \sigma \rVert x(k)-\overline{x}(k) \rVert,
		\end{split}
	\end{equation}
	where $\sigma \in (0,1)$. This means that $x(k)$  converges to  $\overline{x}(k)$ as $k \rightarrow \infty$. Let $y(k)=\overline{x}(k)$, next we will show that $y(k)$ converges to $x^*$. Using $y(k)$ in  the update iteration in \eqref{sdk} , we have
	\begin{equation}
		{y(k+1)} = y(k)-[I-((\Gamma_P+h(k))^{-1}R)^{k+1}] h(k)^{-1}g(k), \label{26}
	\end{equation}
{ which} is consistent with the form of Algorithm 1 in {\cite{wang2024superlinear}}.
{ Thus, there holds}
	\begin{equation}
		{F(y(k+1))-F(x^*)} \leq c(F(y(k))-F(x^*)). \label{27}
	\end{equation}
	where $0 \leq c<1$. {  More details can refer to} the proof of Lemma 2 in {\cite{wang2024superlinear}}. This proof is completed.
\end{proof}

Now we are in a position to discuss the convergence rate of DOCMC algorithm.
\begin{theorem}\label{thmmain0}
Under the condition of Assumption \ref{assm2}, when $\{x(k)\}$ generated by DOCMC algorithm {  is convergent, then} there exists a scalar $r_1>0$ such that
	\begin{align}
		||x(k+1)-x^*||&\le r_1 \rho((\Gamma_P+h(k)))^{-1}\Gamma_P))^{k+1} \\
		&\times ||x(k)-x^*||, \label{conrate2}
	\end{align}
{  that is to say},  DOCMC algorithm is superlinearly convergent.
\end{theorem}
\begin{proof}
{  Denote
\begin{eqnarray}
 z_k(x(k))&\hspace{-0.8em}=&\hspace{-0.8em}-[I-((\Gamma_P+h(k))^{-1}\Gamma_P)^{k+1}]h(k)^{-1}g(k)\nonumber\\
 &\hspace{-0.8em}&\hspace{-0.8em}-[(\Gamma_P+h(k))^{-1}\Gamma_P]^{k+1} \Gamma_P^{-1}B^TPe(k).\label{zkxk}
\end{eqnarray}}

{It is evident from \eqref{zkxk} that }
	\begin{align}
		z_k(x^*)=&-[I-((\Gamma_P+h^*)^{-1}\Gamma_P)^{k+1}]{h^*}^{-1}g^*\nonumber\\
		&-[(\Gamma_P+h^*)^{-1}\Gamma_P]^{k+1} \Gamma_P^{-1}B^TPe^*\nonumber\\
		=&0, \label{barzxstar}\\
		z_k'(x^*)=&-(I-((\Gamma_P+{h^*}))^{-1}\Gamma_P)^{k+1})\nonumber\\
		&-[(\Gamma_P+h^*)^{-1}\Gamma_P]^{k+1} \Gamma_P^{-1}B^TPB. \label{debarzxstar}
	\end{align}
	From \eqref{sdk}, {  one has}
	\begin{align}
		&x(k+1)-x^*\notag\\
		=&x(k)-x^*+z_k(x(k))\notag\\
		\approx&x(k)-x^*+[z_k(x^*)+z_k'(x^*)(x(k)-x^*)]\notag\\
		=&x(k)-x^*+ z_k'(x^*)(x(k)-x^*)\notag\\
		=&[((\Gamma_P+h^*))^{-1}\Gamma_P)^{k+1}\Gamma_P^{-1}R](x(k)-x^*).
	\end{align}
	{ Noting} that $\Gamma_P$ is positive definite, { there exists a matrix $C$ satisfying} $\Gamma_P=C'C$. {  To this end, it generates}
{	$(\Gamma_P+h^*))^{-1}\Gamma_P=(\Gamma_P+h^*))^{-1}C'C=C^{-1}C(\Gamma_P+h^*))^{-1}C'C$}. The second equality indicates $(\Gamma_P+h*))^{-1}\Gamma_P=\bar C^{-1} \bar \Lambda \bar C$, where $\bar \Lambda$ is a diagonal matrix. Due to
$\rho((\Gamma_P+h^*))^{-1}\Gamma_P)<1$, it {  immediately gets} $\rho(\bar \Lambda)=\rho((\Gamma_P+h^*))^{-1}\Gamma_P)<1$ and
	\begin{align}
		&||x(k+1)-x^*||\notag\\
		=&||\bar C^{-1} \bar \Lambda^{k+1} \bar C \Gamma_P^{-1}R(x(k)-x^*)||\notag\\
		\le &\rho((\Gamma_P+h^*))^{-1}\Gamma_P)^{k+1}||\bar C^{-1}||_m\ ||\bar C||_m\ \nonumber\\
		&\times ||\Gamma_P^{-1}R||_m\ ||x(k)-x^*||.
	\end{align}
{ By	letting $r_1=||\bar C^{-1}||_m||\bar C||_m||\Gamma_P^{-1}R||_m\ $, \eqref{conrate2} is obtained accordingly}. The proof is completed.
	
\end{proof}
Now,  { the superlinear convergence of DOAOC algorithm is proved.}

\begin{theorem}\label{thmmain}
Under the condition of Assumption \ref{assm2}, selecting $0<\eta < 1$ and $c=\|I-\eta h^*\|<1$,
when $\{x(k)\}$ generated by DOAOC algorithm converges, then there exists a scalar $r_2>0$ such that
	\begin{align}
		\|x(k+1)-x^*\|\le r_2 c^{k}\|x(k)-x^*\|. \label{conrate3}
	\end{align}
{  That is to say,}  DOAOC algorithm  is superlinearly convergent.
\end{theorem}
\begin{proof}
{ Following from the d}irect derivation{ , it} shows
	\begin{align}
		\bar d_k(x^*)=&-[I-(I-\eta h^*)^{k+1}]{h^*}^{-1}g^*\nonumber\\
		&-((I-\eta h^*)^{k}\eta B^TP)e^*=0, \label{hatgkxstar}\\
		\bar d'_k(x^*)=&-(I-(I-\eta h^*)^{k+1})-\eta(I-\eta h^*)^kB^TPB. \label{dehatgkxstar}
	\end{align}
	Let $I-(\eta h^*+B^TPB)=G$. From \eqref{ocpl}, it follows that
	\begin{align}
		&x(k+1)-x^*\notag\\
		=&x(k)-x^*+\bar d_k(x(k))\notag\\
		\approx&x(k)-x^*+[\bar d_k(x^*)+\bar d'_k(x^*)(x(k)-x^*)]\notag\\
		=&x(k)-x^*+\bar d'_k(x^*)(x(k)-x^*)\notag\\
		=&(I+\bar d'_k(x^*))(x_k-{x^*})\notag\\
		=&((I-\eta h^*)^{k+1}-\eta(I-\eta h^*)^{k}B^TPB)(x(k)-{x^*})\notag\\
		=&[(I-\eta h^*)^kG](x(k)-{x^*}).
	\end{align}
	Since $\|I-\eta h^*\|<1$, when $\eta<2/m_2$  it follows that $ \|x(k+1)-x^*)\|\le \|I-\eta h^*\|^k\|G\|_m \|x(k)-x^*)\|$. { By letting} $r_2=\|G\|_m$, { then the superlinearly convergence of} the DOAOC algorithm  { can be obtained}. The proof is completed.
\end{proof}

\section{CONCLUSIONS}
This paper has proposed  distributed second-order optimization algorithms  with global superlinear convergence {  from the viewpoint of} optimal control theory. It successfully integrates second-order information into distributed optimization without requiring the inversion of the Hessian matrix. A connection has been { revealed} between algorithms and traditional optimization methods.  The convergence analysis of algorithms also have been achieved.

\end{document}